\documentclass[a4paper,12pt]{amsart}
\pdfoutput=1
\usepackage{etex}
\usepackage[a4paper,margin=1in]{geometry}
\usepackage{amssymb,amsfonts,amsmath,amsthm}
\usepackage{latexsym,amsrefs,mathtools,stmaryrd}

\def\bra#1#2{\ensuremath{\langle#1\rangle}_{#2}}

\newcommand{\CC}{\mathbb{C}}
\newcommand{\ZZ}{\mathbb{Z}}
\newcommand{\QQ}{\mathbb{Q}}

\newcommand{\RR}{\mathbb{R}}

\renewcommand{\phi}{\varphi}

\title{Unit distance graphs and algebraic integers}
\author{Danylo~Radchenko}
\address{Max Planck Institute for Mathematics\\
    Vivatsgasse 7\\
    53111 Bonn, Germany}
\email{danradchenko@gmail.com}

\newtheorem{theorem}{Theorem}

\newtheorem{problem}{Problem}

\begin{document}
\begin{abstract}
    We answer a question of Brass about vertex degrees in
    unit distance graphs of finitely generated additive subgroups of $\RR^2$.
\end{abstract}
\maketitle
The following problem was posed in~\cite{BrMoPa}*{p.~186} (see also~\cite{Br}).
\begin{problem}
    Does there exist a finitely generated additive subgroup 
    $\mathcal{A}\subset\RR^2$ such 
    that there are infinitely many elements of $\mathcal{A}$ 
    lying on the unit circle?
\end{problem}
If we look at the graph whose vertex 
set is a given finitely generated additive 
group~$\mathcal{A}\subset\RR^2$, and 
where $x,y\in\mathcal{A}$ are joined by an edge if 
and only if $|x-y|=1$, then the question is whether such a 
graph can have infinite vertex degrees. This question
was motivated by the problem of Erd\H{o}s about maximal number 
of unit distances among $n$ points in the plane~\cite{Er}.
In this short note we give a positive answer to Problem~1. 
Let us denote 
    \[\bra{v_1,\dots,v_n}{\ZZ}:=
    \{a_1v_1+\dots+a_nv_n\;|\;a_1,\dots,a_n\in\ZZ\}.\]

\begin{theorem} \label{thm:example}
    There exist four vectors $v_1,v_2,v_3,v_4\in\RR^2$ such that 
    $\mathcal{A}=\bra{v_1,v_2,v_3,v_4}{\ZZ}$ has infinitely many
    elements on the unit circle.
\end{theorem}
\begin{proof}
    We identify $\RR^2$ with $\CC$. Let $p\in\ZZ[z]$ be the irreducible polynomial
        \[p(z) = z^4-z^3-z^2-z+1.\]
    Using Descartes' rule of sign and noting that $p(1)<0$ we see that
    $p$ has exactly two positive real roots. Since $z^4-z^2+1>0$ for all $z\in\RR$,
    it has no negative roots.
    Since $p$ is reciprocal, the two remaining complex roots $\alpha$ 
    and $\overline\alpha$ must be inverses of each other, 
    so that $|\alpha|=1$. Next, we have 
    $\alpha^m\in\bra{ 1,\alpha,\alpha^2,\alpha^3}{\ZZ}$ for all $m\in\ZZ$ because~$p$ is monic and has 
    integer coefficients. Moreover, the numbers 
    $\alpha^m$, $m\in\ZZ$ satisfy $|\alpha^m|=1$ and
    are all distinct since~$p$ is not divisible by any 
    cyclotomic polynomial. Therefore, the set
    $\mathcal{A}=\bra{1,\alpha,\alpha^2,\alpha^3}{\ZZ}$ 
    has infinitely many elements on the unit circle.
\end{proof}

Of course, the above construction works for any
algebraic integer $\alpha\in\CC$ with~$|\alpha|=1$.
If we take $\alpha$ to be a primitive $m$-th root of unity, then 
the resulting unit distance graph has degree $m$ or $2m$ 
depending on the parity of~$m$.
Taking $\alpha\in\CC$ to be any other algebraic integer of absolute value~$1$
gives a unit distance graph of infinite degree. For example,
one can take any non-real conjugate of a Salem number
(a positive algebraic integer, all of whose conjugates have absolute value $\le1$, 
see~\cite{Sal}*{Ch.~3}); this is exactly what was done in the proof above.

It is also natural to ask what is the minimal possible 
number of generators for such~$\mathcal{A}$. 
The following result shows that in this sense 
Theorem~\ref{thm:example} is optimal.
\begin{theorem}\label{thm:rank3}
    If an additive subgroup $\mathcal{A}\subset\RR^2$ 
    has rank $\le 3$, then it has
    only finitely many elements on the unit circle $S^1$.
\end{theorem}
\begin{proof}
    Suppose that $\mathcal{A}\cap S^1$ is infinite. 
    We can restrict to the case when $\mathrm{rk}(\mathcal{A})=3$
    and $\mathcal{A}$ spans $\RR^2$ over reals, since otherwise~$\mathcal{A}$ 
    is contained in a lattice or in a line and thus 
    intersects~$S^1$ in a finite set. By the same 
    argument, the subgroup generated by $\mathcal{A}\cap S^1$ 
    must also have rank~$3$, and therefore we can choose 
    three elements $v_1,v_2,v_3\in\mathcal{A}\cap S^1$ and 
    an integer $n>0$ such that 
    $\mathcal{A} \subseteq \bra{v_1/n,v_2/n,v_3/n}{\ZZ}$.
    
    Let us denote $(v_1,v_2)=\gamma$, $(v_2,v_3)=\alpha$, 
    and $(v_3,v_1)=\beta$. Before we continue, let us 
    record some elementary identities between $v_i$ and the 
    inner products $\alpha,\beta,\gamma$. 
    First, since the vectors $v_i$ are in $\RR^2$, the determinant of the Gram 
    matrix of $v_1,v_2,v_3$ is $0$, that is
        \begin{equation}\label{eq:gram}
        1+2\alpha\beta\gamma-\alpha^2-\beta^2-\gamma^2 = 0.
        \end{equation}
    Next, since $v_i$ are linearly independent over $\QQ$, we have $|\alpha|,|\beta|,|\gamma|<1$,
    and thus, using~\eqref{eq:gram}, we see that $\alpha\ne\beta\gamma$, $\beta\ne\alpha\gamma$, 
    and $\gamma\ne\alpha\beta$. Finally, we have the following identity
        \begin{equation}\label{eq:v1v2v3}
        \frac{v_1}{\alpha-\beta\gamma}
        +\frac{v_2}{\beta-\alpha\gamma}
        +\frac{v_3}{\gamma-\alpha\beta} = 0.
        \end{equation}
    
    Let $Q(x,y,z)=\|xv_1+yv_2+zv_3\|^2$ and let $\mathcal{S}$ 
    be the set of solutions $(x,y,z)\in\ZZ^3$ 
    of equation $Q(x,y,z)=n^2$; by our assumption this set is infinite. A triple~$(x,y,z)$ belongs to $\mathcal{S}$ 
    if and only if
        \[x^2+y^2+z^2+2\alpha yz+2\beta zx+2\gamma xy = n^2.\]
    
    Since $v_1$, $v_2$, and $v_3$ are independent over $\QQ$, the intersection of $\mathcal{S}$ with any plane in $\ZZ^3$ is finite. In particular, the
    set $\mathcal{S}_0=\{(x,y,z)\in\mathcal{S}\,|\, xyz\ne 0\}$
    is infinite, and the condition 
    for $(x,y,z)\in\mathcal{S}_0$ can be rewritten as
        \begin{equation}\label{eq:rationality}
        \frac{\alpha}{x}+\frac{\beta}{y}+\frac{\gamma}{z} 
        = \frac{n^2-x^2-y^2-z^2}{2xyz}.
        \end{equation}
    Denote $\mathcal{V}=\{(1/x,1/y,1/z)\;|\; (x,y,z)\in\mathcal{S}_0\}$.
    We consider three cases.
    
    \textbf{Case 1.} The $\QQ$-span of $\mathcal{V}$ 
    is $\QQ^3$. By choosing any three linearly independent 
    vectors in $\mathcal{V}$ and solving~\eqref{eq:rationality} 
    for $\alpha,\beta,\gamma$, we get that $\alpha$, $\beta$, 
    and $\gamma$ must be rational. But then~\eqref{eq:v1v2v3}
    gives a linear relation over $\QQ$ contradicting 
    the fact that $\mathrm{rk}(\mathcal{A})=3$.
    
    \textbf{Case 2.} The $\QQ$-span of $\mathcal{V}$ 
    is $1$-dimensional. 
    Then $\mathcal{S}_0$ is contained in a line through the 
    origin, and by homogeneity of $Q$ we get 
    that $|\mathcal{S}_0|\le 2$, a contradiction.
    
    \textbf{Case 3.} The $\QQ$-span of $\mathcal{V}$ 
    is $2$-dimensional. Then there is a unique nonzero 
    triple $(a,b,c)\in\ZZ^3$ such that for all $(x,y,z)\in\mathcal{S}_0$ 
    we have $a/x+b/y+c/z=0$. In this case, by
    solving~\eqref{eq:rationality} for $\alpha,\beta,\gamma$, we get
        \begin{equation}\label{eq:alphas}
        \alpha=\lambda a+\alpha_0,\quad \beta
        =\lambda b+\beta_0,\quad\gamma=\lambda c+\gamma_0
        \end{equation}
    for some $(\alpha_0,\beta_0,\gamma_0)\in\QQ^3$ and $\lambda\in\RR$.
    We may assume that $abc\ne 0$,
    since otherwise $\mathcal{S}_0$ would be contained in one 
    of the three planes $ay+bx=0$, $az+cx=0$, or $bz+cy=0$, 
    and hence would be finite.
    
    Next, recall that all integer solutions of $1/x+1/y+1/z=0$ 
    are given by $x=dst$, $y=drt$, $z=drs$ for 
    some nonzero integers $d,r,s,t$ with $r+s+t=0$ (simply note that
    the equation is equivalent to $z^2=(z+x)(z+y)$ that is in turn equivalent to $z^2=x'y'$).
    By applying this observation to the equation $a/x+b/y+c/z=0$ 
    we deduce that any solution $(x,y,z)\in\mathcal{S}_0$ 
    can be written as
        \[x=\frac{dst}{bc},\quad y=\frac{drt}{ac},\quad z=\frac{drs}{ab}\]
    for some nonzero integers $d,r,s,t$ with $r+s+t=0$.
    Plugging these expressions back into $Q$ and setting $v_3=\kappa_1v_1+\kappa_2v_2$ 
    (where $\kappa_i$ can be computed from~\eqref{eq:v1v2v3}) we get
        \[n^2 = \|xv_1+yv_2+zv_3\|^2 = 
        \frac{d^2}{(abc)^2}\|s((\kappa_1c-a)r-as)v_1+r((\kappa_2c-b)s-br)v_2\|^2,\]
    and thus, setting $\phi_1(r,s)=(\kappa_1c-a)r-as$ and
    $\phi_2(r,s)=(\kappa_2c-b)s-br$, we get 
        \[\|s\phi_1(r,s)v_1+r\phi_2(r,s)v_2\|^2
        = \frac{(abcn)^2}{d^2} \le (abcn)^2.\]
    Since $v_1$ and $v_2$ form a basis in $\RR^2$ and $|s|,|r|\ge 1$ 
    we get that $|\phi_i(r,s)|\le C$
    for some constant $C$ that depends only on $a,b,c,n$, and $\gamma$. 
    If $\phi_1$ and $\phi_2$ are not proportional, 
    then these two inequalities define a bounded region 
    in the $(r,s)$-plane, and hence there are only finitely many
    solutions in this case, contrary to our assumption.
    Therefore,~$\phi_1$ and~$\phi_2$ must be proportional, that is
        \[(\kappa_1c-a)(\kappa_2c-b)-ab = 0\]
    or equivalently
        \begin{equation}\label{eq:secondrelation}
        a(\alpha-\beta\gamma)+b(\beta-\alpha\gamma)+c(\gamma-\alpha\beta) = 0.
        \end{equation}
    
    If we now substitute~\eqref{eq:alphas} into~\eqref{eq:gram} 
    and into~\eqref{eq:secondrelation},
    we obtain two different polynomial equations on $\lambda$, 
    $p_1(\lambda)=0$ and $p_2(\lambda)=0$, where
        \begin{multline*}
        p_1(\lambda) \;=\; 1+2\alpha\beta\gamma-\alpha^2-\beta^2-\gamma^2
        = 2abc\lambda^3
        +(2bc\alpha_0+2ac\beta_0+2ab\gamma_0-a^2-b^2-c^2)\lambda^2\\
        +2(a\beta_0\gamma_0+b\alpha_0\gamma_0+c\alpha_0\beta_0
        -a\alpha_0-b\beta_0-c\gamma_0)\lambda
        +(2+\alpha_0\beta_0\gamma_0-\alpha_0^2-\beta_0^2-\gamma_0^2)
        \end{multline*}
    and 
        \begin{multline*}
        p_2(\lambda) \;=\; a(\alpha-\beta\gamma)+b(\beta-\alpha\gamma)+c(\gamma-\alpha\beta) = -3abc\lambda^2 \\
        -(2bc\alpha_0+2ac\beta_0+2ab\gamma_0 - a^2-b^2-c^2)\lambda -(a\beta_0\gamma_0+b\alpha_0\gamma_0+c\alpha_0\beta_0-a\alpha_0-b\beta_0-c\gamma_0).
        \end{multline*}
    We observe that $2p_2(\lambda)=-p_1'(\lambda)$, so 
    that $\lambda$ is a double root of a cubic 
    polynomial~$p_1$ with rational coefficients. 
    It is well-known that this implies $\lambda\in\QQ$, 
    but then $\alpha$, $\beta$, and $\gamma$ are also 
    rational, and once again we arrive at a contradiction.
\end{proof}

\end{document}